\documentclass[12pt]{amsart}
\usepackage{fullpage} 
\usepackage[normalem]{ulem}
\usepackage[T1]{fontenc}
\usepackage[utf8]{inputenc}

\usepackage{amsfonts,amsmath,amsthm,amssymb}
\usepackage{enumitem} 
\usepackage{siunitx} 

\usepackage[hidelinks,linktoc=all]{hyperref}

\numberwithin{table}{section}
\numberwithin{figure}{section}
\numberwithin{equation}{section}
 \usepackage[backend=biber,style=numeric,hyperref=true,isbn=false,doi=false]{biblatex}
\usepackage{pgf} 
\usepackage{graphicx}
\usepackage{subcaption}
\usepackage{adjustbox}
\DeclareMathOperator{\Frobeniusoper}{F} 
\DeclareMathOperator{\multiplicityoper}{m} 
\DeclareMathOperator{\leftsoper}{L} 
\DeclareMathOperator{\extleftsoper}{L^{\prime}} 
\DeclareMathOperator{\primitivesoper}{P} 
\DeclareMathOperator{\genusoper}{g} 

\DeclareMathOperator{\depth}{\mathrm{q}}
\DeclareMathOperator{\pdepth}{\mathrm{\pi}} 

\DeclareMathOperator{\suchthat}{ : } %

\newcommand{\mainmap}{\(\Phi\)}

\newcommand{\maxprimset}[1]{\mathcal{A}_{#1}}
\newcommand{\maxprimcard}[1]{{A}_{#1}}

\newcommand{\frobset}[1]{\mathcal{N}_{#1}}
\newcommand{\frobcard}[1]{{N}_{#1}}

\newtheorem{theorem}{Theorem}[section] 
\newtheorem*{theorem*}{Theorem}
\newtheorem{lemma}[theorem]{Lemma}
\newtheorem{corollary}[theorem]{Corollary}
\newtheorem{proposition}[theorem]{Proposition}
\newtheorem*{proposition*}{Proposition}

\theoremstyle{definition}
\newtheorem{definition}[theorem]{Definition}

\newtheorem{example}[theorem]{Example}
\newtheorem{problem}[theorem]{Problem}

\newtheorem{remark}[theorem]{Remark}

\usepackage{colortbl}

\title{On counting numerical semigroups by maximum primitive and Wilf's conjecture}
\author{Manuel Delgado, Neeraj Kumar}
\address{CMUP--Centro de Matemática da Universidade do Porto, Departamento de Matemática, Faculdade de Ciências, Universidade do Porto, Rua do Campo Alegre s/n, 4169– 007 Porto, Portugal} 
\email{mdelgado@fc.up.pt, up202011201@fc.up.pt} 

\author{Claude Marion}
\address{Instituto de Matemática, Estatística e Ciência da Computação, Universidade de São Paulo, Rua do Matão 1010, São Paulo, SP 005508-090, Brazil}
\email{marion@ime.usp.br}

\thanks{The authors were partially supported by CMUP, member of LASI, which is financed by national funds through FCT – Fundação para a Ciência e a Tecnologia, I.P., under the projects with reference UID/00144/2025 and associated DOI given by https://doi.org/10.54499/UID/00144/2025.\\
  The first author also acknowledges the Proyecto de Excelencia de la Junta de Andalucía (ProyExcel 00868).\\
  The second author acknowledges the support from FCT through the PhD Scholarship UI/BD/150941/2021} 

\date{\today}
\begin{document}



\begin{abstract}
  We introduce a new way of counting numerical semigroups, namely by their maximum primitive, and show its relation with the counting of numerical semigroups by their Frobenius number. We show that these two ways of counting are Möbius transforms of one another. We also establish that almost all numerical semigroups with large enough maximum primitive satisfy Wilf's conjecture. A crucial step in the proof is a result of independent interest: a numerical semigroup \(S\) with multiplicity \(\multiplicityoper{}\) such that \(|S\cap (\multiplicityoper{},2 \multiplicityoper{})|\geq \sqrt{2\multiplicityoper{}} \) satisfies Wilf's conjecture.

\end{abstract}

\maketitle

\section{Introduction}\label{sec:introduction}

Let \(\mathbb{N}\) denote the set of nonnegative integers. A \emph{numerical semigroup} \(S\) is a subsemigroup of \(\mathbb{N}\) that contains \(0\) and is cofinite. 
It is well known that for every numerical semigroup \(S\) there exists a unique finite set that generates \(S\) and is minimal under inclusion among all the sets that generate \(S\). We refer to this set as the \emph{minimal set of generators} and we denote it by \(\primitivesoper(S)\). 
The elements of \(\primitivesoper(S) \) are called the \emph{minimal generators} or \emph{primitive elements} or \emph{primitives} of \(S\). The smallest primitive is called the \emph{multiplicity} of \(S\) and is denoted by \(\multiplicityoper(S)\). The cardinality of \(\primitivesoper(S)\) is called the \emph{embedding dimension} of \(S\), and is denoted by \(e(S)\). 
A positive integer that does not belong to a numerical semigroup \(S\) is called a \emph{gap} of \(S\), and the number of gaps of \(S\) is called the \emph{genus} of \(S\) and is denoted by \(\genusoper(S)\). For \(S\neq \mathbb{N}\), the largest gap of \(S\) is called the \emph{Frobenius number} of \(S\) and is denoted by \(\Frobeniusoper(S)\). Note that \(\mathbb{N}\) has no gaps, and by convention \(\Frobeniusoper(\mathbb{N})\) is set to be \(-1\). An element of \(S\) that is smaller than \(\Frobeniusoper(S)\) is called a \emph{left element} of \(S\), and we denote the set of left elements of \(S\) by \(\leftsoper(S)\). The integer \(\lceil(\Frobeniusoper(S)+1)/{\multiplicityoper(S)}\rceil\) is known as the \emph{depth} of \(S\) and is denoted \(\depth(S)\). 

Let \(n\) be a positive integer. We define \(\maxprimset{n}\) as the set of numerical semigroups with maximum primitive equal to \(n\), and we denote the cardinality of \(\maxprimset{n}\) by \(\maxprimcard{n}\).
We refer to the problem of determining \(\maxprimcard{n}\) as the problem of \emph{counting by maximum primitive}, which is the main topic of this article. 
We define \(\frobset{n}\) as the set of numerical semigroups with Frobenius number~\(n\), and we denote the cardinality of \(\frobset{n}\) by \(\frobcard{n}\). The problem of computing \(\frobcard{n}\) is referred to as the problem of \emph{counting numerical semigroups by Frobenius number} or \emph{counting by Frobenius number}. For an analogy, we remark that counting by Frobenius number could also be referred to as \emph{counting by maximum gap}. 
In 1978, Wilf~\cite{Wilf1978AMM-circle} posed the problem of determining the asymptotic growth of \(\frobcard{n}\). In 1990, Backelin solved the latter problem (see~\cite[Proposition 1]{Backelin1990MS-number}) and provided bounds for \(\frobcard{n}\).
  Although \((\frobcard{n})\) is not a monotonic sequence (see Table~\ref{table:counting_by_maxprim-and-frob}), Singhal~\cite{Singhal2022SF-Distribution} showed that \(\frobcard{n+2}\geq \frobcard{n}\) for all \(n\). The analogous result for \((\maxprimcard{n})\), namely that \(\maxprimcard{n+2} \geq \maxprimcard{n} \) for all \(n\), was showed by the second author in \cite{Kumar2025CiA-Monotonicity}.
The problem of counting by Frobenius number and related asymptotic problems have regained interest recently (see~\cite{Singhal2022SF-Distribution,Li2023CT-Counting}). The computational estimation of \(\frobcard{n}\) has been studied in~\cite{BrancoOjedaRosales2021PM-set}, and the initial values of  the sequence \((\frobcard{n})\) constitute the OEIS~\cite{OEIS2026} sequence A124506. 

  Let \(g\) be a positive integer. We denote the number of numerical semigroups with genus \(g\) by \(n_g\). The problem of \emph{counting by genus} refers to the problem of computing \(n_g\) (see \cite{Bras-Amoros2008SF-Fibonacci, FromentinHivert2016MC-Exploring, Kaplan2017AMM-Counting,  Zhai2012SF-Fibonacci}) and it corresponds to the OIES~\cite{OEIS2026} sequence A007323. Bras-Amorós~\cite{Bras-Amoros2008SF-Fibonacci} conjectured that \(n_{g+1}\geq n_{g}+n_{g-1}\) for all \(g\), which remains an open problem. Moreover the weaker conjecture that \(\left(n_g\right) \) is a monotonic sequence remains unresolved. Zhai~\cite{Zhai2012SF-Fibonacci} proved that the ratio \(n_{g+1}/(n_{g}+n_{g-1})\) tends to \(1\) as \(g \to \infty\),
  thereby showing that the sequence \((n_g)\) is monotonically increasing for all \(g\) greater than some integer~\(C\). 

An open problem posed by Wilf~\cite{Wilf1978AMM-circle}, now widely referred to as Wilf's conjecture, has received considerable attention. Wilf asked whether the following inequality holds for every numerical semigroup \(S\):
\[|\primitivesoper(S)|\cdot |\leftsoper(S)|\geq \Frobeniusoper(S)+1.\]

Kaplan (\cite[Proposition~26]{Kaplan2012JPAA-Counting}) showed that the numerical semigroups with depth at most~\(2\) satisfy Wilf's conjecture. Eliahou established the corresponding result for numerical semigroups with depth less than or equal to 3 in~\cite{Eliahou2018JEMS-Wilfs}. These two results along with a result of Zhai (see~\cite{Zhai2012SF-Fibonacci}) establish that the proportion of numerical semigroups with genus \(g\) that satisfy Wilf's conjecture tends to \(1\) as \(g\) tends to infinity, as remarked by Eliahou in~\cite{Eliahou2018JEMS-Wilfs}. For other known results related to Wilf's conjecture see the survey by the first author~\cite{Delgado2020-survey}.\\

In this article we first obtain an equivalence between \(\maxprimcard{n}\) and \(\frobcard{n}\) viewed as functions in~\(n\). More precisely, we prove that the sequences \((\maxprimcard{n})\) and \((\frobcard{n})\) are Möbius transforms of one another.

\begin{theorem}\label{th:frob_in_terms_of_max_prim} 
	Let \(n\) be a positive integer. Then
	\[\frobcard{n} = \sum_{d| n} \maxprimcard{\frac{n}{d}}.\]
	In particular, for a prime \(p\), \(\frobcard{p}=\maxprimcard{p}+1\).
\end{theorem}

\begin{corollary}\label{cor:dependence_of_Af_and_Nf} 
	Let \(n\) be a positive integer. Then
	\[ \maxprimcard{n} =\sum_{d|n}\mu\left(\frac{n}{d}\right)\cdot \frobcard{d}\]
	where \(\mu\) is the M\"obius function.
\end{corollary}

We then show that \(\maxprimcard{n}\) and \(\frobcard{n}\) are asymptotically equivalent. In the result below we use the following small-$o$ notation. For functions \(A,B: \mathbb{N} \to \mathbb{N}\) we write \(A(n)=o(B(n))\) when \(\lim\limits_{n\to\infty} \frac{A(n)}{B(n)}=0\) and \(B(n)\neq 0\) for all \(n> C\), for some \(C\in\mathbb{N}\). Given a function \(D: \mathbb{N} \to \mathbb{N}\) we write \(D(n) = n^{o(1)} \) to denote that \(\log (D(n)) = o(\log n) \). 

\begin{theorem}\label{th:asymptotics_maxprim_counting}
	Consider the functions \(A, N: \mathbb{N}\to\mathbb{N}\) given by \(A(0)=N(0)=0\), and \(A(n)=\maxprimcard{n}\) and \(N(n)=\frobcard{n}\)  for \(n>0\). Then 
	\[ N(n)- A(n)  =  n^{o(1)} 2^{\frac{n}{4}}.
	\]
\end{theorem}

We thus obtain the following consequence of Theorem~\ref{th:asymptotics_maxprim_counting}. 
\begin{corollary}\label{cor:subsequence_max_prim_count} 
	The following assertion holds:
	\[\lim_{n\to\infty} \frac{\maxprimcard{n}}{\frobcard{n}}= 1.\]
	In particular the sequences \((\maxprimcard{2n})\) and \((\maxprimcard{2n+1})\) grow asymptotically as constant multiples of~\(2^{n}\). 
\end{corollary}

We next consider Wilf's conjecture, and prove that it holds under the following criterion.
\begin{theorem}\label{th:Wilfs_conjecture_new_criterion}
  Let \(S\) be a numerical semigroup with multiplicity \(m\) and depth \(d\). If     \(d> 2\)  and
	\[|S\cap (m,2m)|\geq \sqrt{\frac{dm}{d-2}} \]
	then \(S\) satisfies Wilf's conjecture.
    In particular, if \(|S\cap (m,2m)|\geq \sqrt{2m}\) then \(S\) satisfies Wilf's conjecture.  
\end{theorem}

	This yields the following improvement to a result of Kaplan (see Theorem \ref{th:Kaplan_genus_Wilf}).

\begin{corollary}\label{cor:improvement_on_Wilf_2g_3m_criterion}
	Let \(S\) be a numerical semigroup with genus \(g\) and multiplicity \(m\). If \(g\) is less than \(2m-\sqrt{2m}\) then \(S\) satisfies Wilf's conjecture.
\end{corollary}

We establish as a consequence of Theorem~\ref{th:Wilfs_conjecture_new_criterion} that the proportion of numerical semigroups with maximum primitive $n$ that satisfy Wilf's conjecture tends to \(1\) as $n$ tends to infinity.
 
\begin{theorem}\label{th:Asymptotic_wilf_for_max_prim_counting}
	The following assertion holds:
	\[  \lim_{n\to\infty} \frac{\left|\{S\in \maxprimset{n} : |\primitivesoper(S)|\cdot|\leftsoper(S)|\geq \Frobeniusoper(S)+1 \}\right|}{\maxprimcard{n}} = 1. \]
\end{theorem}

We then discuss how Theorem~\ref{th:Wilfs_conjecture_new_criterion} relates to some other known results on Wilf's conjecture. We identify an interesting class of numerical semigroups and pose the following problem. 
\begin{problem}\label{prob:primitve_depth_2_wilf}
	Let \(S\) be a numerical semigroup with multiplicity \(m\) and maximum primitive \(n\). If \(n< 2m\) does \(S\) satisfy Wilf's conjecture? 
\end{problem}

The outline of the paper is as follows. In Section~\ref{sec:preliminaries} we recall some results on counting by Frobenius number and by genus, and give some elementary results related to counting by maximum primitive and counting by Frobenius number. Section~\ref{sec:mainmap} is devoted to the study of a map~\(\Phi\) that allows us to relate the sets \(\maxprimset{n}\) and \(\frobset{n}\). 
In Section~\ref{sec:Nf-Af} we show the equivalence of counting by maximum primitive and counting by Frobenius number, proving  Theorems \ref{th:frob_in_terms_of_max_prim} and \ref{cor:dependence_of_Af_and_Nf}. Moreover we study the asymptotic growth of the sequence \( (\maxprimcard{n}) \) and prove 
Theorem \ref{th:asymptotics_maxprim_counting} and Corollary~\ref{cor:subsequence_max_prim_count}. Section~\ref{sec:wilf} is devoted to the conjecture of Wilf and its asymptotic version with respect to counting by maximum primitive. In particular we prove Theorems~\ref{th:Wilfs_conjecture_new_criterion} and~\ref{th:Asymptotic_wilf_for_max_prim_counting}, and Corollary~\ref{cor:improvement_on_Wilf_2g_3m_criterion}.

\section{Preliminaries}\label{sec:preliminaries}
We begin with a few definitions. Given any interval \(I\subseteq \mathbb{R}\), we adopt the convention that \(I\) denotes \(I\cap \mathbb{Z}\) unless specified otherwise. In particular, given two real numbers \(a\) and \(b\) we use the notation \([a,b]\), \((a,b)\), \([a,b)\) and \((a,b]\) to respectively denote the set of integers in those intervals. Unless otherwise stated every integer we consider is a nonnegative integer.
Given a real number \(x\) and a set \(A\subseteq \mathbb{R}\), we let \(x\cdot A\) denote \(\{xa: a \in A\}\). Moreover we denote \((c\cdot A)\cup B\) by \(c\cdot A \cup B\), for any sets \(A, B \subseteq \mathbb{R}\) and \(c\in\mathbb{R} \).

Analogous to the notion of depth of a numerical semigroup, we define the following notion. 
\begin{definition}\label{def:depth}
	The \emph{primitive depth} of a numerical semigroup~\(S\) is defined as \(  \left\lceil \frac{\max \primitivesoper(S)}{\min\primitivesoper(S)} \right\rceil\), and denoted by \(\pdepth(S)\).
\end{definition}

Recall that the set \(\leftsoper(S)\) of left elements of a numerical semigroup~\(S\) is \(\{s\in S : s < \Frobeniusoper(S)\}\). 
\begin{definition}\label{def:ext_left}
	The \emph{extended set of left elements} of a numerical semigroup~\(S\) is
	\[\extleftsoper(S)= \leftsoper(S)\cup \{\Frobeniusoper(S)\}.\]
\end{definition}

Note that the extended set of left elements \(\extleftsoper(S)\) determines the numerical semigroup~\(S\), unlike the set~\(\leftsoper(S)\).

In Table ~\ref{table:counting_by_maxprim-and-frob} below we collect some experimental data consisting of the first \(63\) elements of the sequences \(\left(\maxprimcard{n}\right)\) and \(\left(\frobcard{n}\right)\).
All elements of the sequence \(\left(\maxprimcard{n}\right)\) in the table were obtained by using the \texttt{GAP} code that is part of the \texttt{numericalsgps}~\cite{NumericalSgps1.4.0} \texttt{GAP} package (from Version~1.4.0), and is distributed by \texttt{GAP}~\cite{GAP4.14.0} (from Version~4.14.0).
Regarding the sequence \(\left(\frobcard{n}\right)\), the first \(39\) elements have been known for some time~\cite{RosalesGarcia-SanchezGarcia-GarciaJimenezMadrid2004JPAA-Fundamental}.
We used Theorem~\ref{th:frob_in_terms_of_max_prim} to complete the table. From the data presented in Table~\ref{table:counting_by_maxprim-and-frob} one notes that for any prime \(p < 63\) we have \(\frobcard{p}=\maxprimcard{p}+1\), which is in fact true for any prime \(p\), as established in Theorem~\ref{th:frob_in_terms_of_max_prim}. 

\begin{table}[h]
  \centering
    \setlength{\tabcolsep}{8pt}
\setlength{\arrayrulewidth}{0.4mm}
\begin{tabular}{| r  r  r  | r  r  r  | r  r  r  |}
  \hline
  \(n\)&\( \maxprimcard{n} \)&\( \frobcard{n}  \) &\(n\)&\( \maxprimcard{n} \)&\( \frobcard{n}  \)  &\(n\) &\( \maxprimcard{n} \) &\( \frobcard{n}  \)  \\ \hline
\rowcolor{black!5}
      1 & 1 & 1 & 22&\num{1862}&\num{1913}&43&\num{4889433}&\num{4889434}\\
2 & 0 & 1 & 23&\num{4095}&\num{4096}&44&\num{4783757}&\num{4785671}\\
\rowcolor{black!5}
3 & 1 & 2 & 24&\num{3530}&\num{3578}&45&\num{9574948}&\num{9575167}\\
4 & 1 & 2 & 25&\num{8268}&\num{8273}&46&\num{9674748}&\num{9678844}\\
\rowcolor{black!5}
5 & 4 & 5 & 26&\num{8069}&\num{8175}&47&\num{19919901}&\num{19919902}\\
6 & 2 & 4 & 27&\num{16111}&\num{16132}&48&\num{18893119}&\num{18896892}\\
\rowcolor{black!5}
7 & 10 & 11 & 28&\num{16163}&\num{16267}&49&\num{40010840}&\num{40010851}\\
8 & 8 & 10 & 29&\num{34902}&\num{34903}&50&\num{39437596}&\num{39445886}\\
\rowcolor{black!5}
9 & 19 & 21 & 30&\num{31603}&\num{31822}&51&\num{78793811}&\num{78794277}\\
10 & 17 & 22 & 31&\num{70853}&\num{70854}&52&\num{78922130}&\num{78930306}\\
\rowcolor{black!5}
11 & 50 & 51 & 32&\num{68476}&\num{68681}&53&\num{162306074}&\num{162306075}\\
12 & 35 & 40 & 33&\num{137339}&\num{137391}&54&\num{155991666}&\num{156008182}\\
\rowcolor{black!5}
13 & 105 & 106&34&\num{140196}&\num{140661}&55&\num{325800242}&\num{325800297}\\
14 & 92 & 103 & 35&\num{292066}&\num{292081}&56&\num{320507004}&\num{320523279}\\
\rowcolor{black!5}
15 & 194 & 200 & 36&\num{269817}&\num{270258}&57&\num{643198150}&\num{643199112}\\
16 & 195 & 205 & 37&\num{591442}&\num{591443}&58&\num{644611930}&\num{644646833}\\
\rowcolor{black!5}
17 & 464 & 465 & 38&\num{581492}&\num{582453}&59&\num{1317118755}&\num{1317118756}\\
18 & 382 & 405 & 39&\num{1155905}&\num{1156012}&60&\num{1269732856}&\num{1269765591}\\
\rowcolor{black!5}
19 & 960 & 961 & 40&\num{1160411}&\num{1161319}&61&\num{2640706082}&\num{2640706083}\\
20 & 877 & 900 & 41&\num{2425710}&\num{2425711}&62&\num{2606696049}&\num{2606766903}\\
\rowcolor{black!5}
21 &\num{1816}&\num{1828}&\num{42}&\num{2285281}&\num{2287203}&63&\num{5228284933}&\num{5228286780}\\

  \hline		                  
\end{tabular}

	\caption{ Counting by maximum primitive and by Frobenius number\label{table:counting_by_maxprim-and-frob}}
\end{table}

We next give some definitions and prove some preliminary results needed for later use. Let \(n\) be a positive integer. We denote by \(O_n\) the numerical semigroup consisting \(0\) and all integers greater than or equal to \(n\). These numerical semigroups are often referred to as \emph{ordinary} semigroups. Note that
\(O_n=\langle n,\ldots, 2n-1 \rangle.\) 
\begin{definition}\label{notation:Nf-d}
	Let \(n\) and \(d\) be positive integers. We define \(\frobset{n}(d)\) as the set
	\[ \frobset{n}(d)=\{S\in \frobset{n}: \gcd(\extleftsoper(S)) = d\}\]
	and let \(\frobcard{n}(d)\) denote the cardinality of \(\frobset{n}(d)\).
\end{definition}
Note that if \(d\) is not a divisor of \(n\) then \(\frobset{n}(d)\) is empty. Thus \(\frobset{n}\) can be written as the following disjoint union
\begin{equation}
\label{eq:partition_of_Frob_counting}
\frobset{n} = \bigcup_{d|n} \frobset{n}(d).
\end{equation}
Moreover \(\frobset{n}(n)= \{O_{n+1}\} \).

\begin{definition}
	Given a numerical semigroup \(S\) and a positive integer \(n\), let \(\frac{S}{n} \) denote the set \(\{x\in\mathbb{N}\suchthat nx\in S\}. \) 
\end{definition}

We note that \(\frac{S}{n}\) is a numerical semigroup (see~\cite[Proposition~5.1]{RosalesGarcia2009Book-Numerical}) that is referred to as the \emph{quotient of S by n}. In particular for \(n>1\), \(\frac{S}{n}\) is not the same as the set \(\frac{1}{n} \cdot S = \left\{\frac{s}{n} : s\in S\right\} \). We give an alternate characterisation of \(\frac{S}{n}\) for the case where \(n\) is a divisor of \(\gcd(\extleftsoper(S))\).

\begin{lemma}\label{lemma:quotient-by-ext-left-divisor}
	Let \(S\) be a numerical semigroup and let \(d\) be a divisor of \(\gcd(\extleftsoper(S))\). Then \(\frac{S}{d}=\frac{1}{d}\cdot  \leftsoper(S)\cup\left(\frac{\Frobeniusoper(S)}{d},\infty\right)\). In particular \(\frac{S}{d}\) is a numerical semigroup with Frobenius number equal to \(\frac{\Frobeniusoper(S)}{d}\).
\end{lemma}
\begin{proof}
	We first prove that \(\frac{S}{d}\) is a subset of \(\frac{1}{d}\cdot  \leftsoper(S)\cup\left(\frac{\Frobeniusoper(S)}{d},\infty\right)\). Let \(x\in\frac{S}{d}\), so that \(dx\in S\). If \(dx\) is an element of \(\leftsoper(S)\) then \(x\) is in \(\frac{1}{d}\cdot  \leftsoper(S)\). Otherwise, \(dx\) is an element of \((\Frobeniusoper(S), \infty)\) and thus \(x\) is in \(\left( \frac{\Frobeniusoper(S)}{d}, \infty\right) \). Therefore \(x\in\frac{1}{d}\cdot  \leftsoper(S)\cup\left(\frac{\Frobeniusoper(S)}{d},\infty\right)\), and thus \(\frac{S}{d}\) is a subset of \( \frac{1}{d}\cdot  \leftsoper(S)\cup\left(\frac{\Frobeniusoper(S)}{d},\infty\right)\). \\ 
	We next prove that \( \frac{1}{d}\cdot  \leftsoper(S)\cup\left(\frac{\Frobeniusoper(S)}{d},\infty\right)\)  is a subset of \(\frac{S}{d}\). Let \(x\in\frac{1}{d}\cdot  \leftsoper(S) \cup\left(\frac{\Frobeniusoper(S)}{d},\infty\right)\). Therefore \(dx\in\leftsoper(S) \cup [\Frobeniusoper(S)+d, \infty) \). Thus in particular \(dx\in S\), and therefore by definition \(x\in\frac{S}{d} \). Thus \( \frac{1}{d}\cdot  \leftsoper(S)\cup\left(\frac{\Frobeniusoper(S)}{d},\infty\right)\) is a subset of \(\frac{S}{d}\).\\
	We therefore conclude that \(\frac{S}{d}\) is equal to \( \frac{1}{d}\cdot  \leftsoper(S)\cup\left(\frac{\Frobeniusoper(S)}{d},\infty\right)\). We thus observe that \(\frac{S}{d}\) contains the interval \(\left(\frac{\Frobeniusoper(S)}{d},\infty\right)\) but it does not contain \(\frac{\Frobeniusoper(S)}{d}\). Thus \(\frac{\Frobeniusoper(S)}{d}\) is the Frobenius number of \(\frac{S}{d}\).
	\end{proof}

\begin{definition}\label{def:map-delta}
	Let \(n\) be a positive integer and let \(d\) be a divisor of \(n\). We define the map \(\delta_{n,d}:\frobset{n}(d)\rightarrow \frobset{\frac{n}{d}}(1)\) by \(\delta_{n,d}(S)= \frac{S}{d}.\)
	We further define the map \( \Delta_n: \frobset{n} \rightarrow \bigcup_{d\mid n} \frobset{\frac{n}{d}}(1) \) by \(\Delta_n(S)=\delta_{n,d}(S)\), where \(d=\gcd (\extleftsoper(S))\). 
\end{definition}
Note that the map \(\Delta_n\) restricted to \(\frobset{n}(1)\) is \(\delta_{n,1}\) which is the identity map on \(\frobset{n}(1)\).

\begin{lemma}\label{lemma:maps-delta-injective}
	Let \(n\) be a positive integer and \(d\) be a divisor of \(n\). Then the maps \(\delta_{n,d}\) and \(\Delta_n\) are bijective maps.
\end{lemma}
\begin{proof}
	We first verify that \(\delta_{n,d}\) is injective. Let \(R\) and \(S\) be elements of \(\frobset{n}(d)\) such that \(\delta_{n,d}(R)= \delta_{n,d}(S)= T\) for some \(T\) in \(\frobset{\frac{n}{d}}(1)\). In other words, \( T=\frac{S}{d}=\frac{R}{d}\). In particular \(\Frobeniusoper(T)= \Frobeniusoper\left(\frac{S}{d}\right)= \Frobeniusoper\left(\frac{R}{d}\right)\), and \( \leftsoper(T)=\leftsoper(\frac{S}{d})=\leftsoper(\frac{R}{d})\). Lemma \ref{lemma:quotient-by-ext-left-divisor} yields \(\leftsoper(\frac{S}{d})= \frac{1}{d}\cdot \leftsoper(S) \) and \(\leftsoper(\frac{R}{d})= \frac{1}{d}\cdot \leftsoper(R) \), and since \(\Frobeniusoper(R)= \Frobeniusoper(S)= n\), \(\Frobeniusoper(\frac{R}{d})=\Frobeniusoper(\frac{S}{d})=\Frobeniusoper(T)=\frac{n}{d}\). Thus \( \leftsoper(S)=\leftsoper(R)=d\cdot \leftsoper(T)\).
	Therefore \(S= \leftsoper(S)\cup (n,\infty)= \leftsoper(R)\cup (n,\infty)= R \) and so \(\delta_{n,d}\) is injective.
	
	We now verify the surjectivity of \(\delta_{n,d}\). Let \(T\in \frobset{\frac{n}{d}}(1)\) and let \(T'\) be the set \(d \cdot \leftsoper(T) \cup (n,\infty)\). We observe that \(T'\) is an element of \(\frobset{n}(d)\) and \(\delta_{n,d}(T')= T\). Thus \(\delta_{n,d}\) is surjective. 
	
	Therefore \(\delta_{n,d}\) is bijective. We note that the inverse of \(\delta_{n,d}\) is given by the map
	\[ S \mapsto  {d}\cdot{\leftsoper(S)} \cup \left( d\cdot\Frobeniusoper(S), \infty \right)   .\]
	
	Since \(\{ \frobset{n}(d): d \text{ divides } n \} \) is a partition of \(\frobset{n}\) (by (\ref{eq:partition_of_Frob_counting})) and each \(\frobset{n}(d)\) is bijectively mapped to the corresponding set \(\frobset{\frac{n}{d}}(1)\) under \(\delta_{n,d}\), \(\Delta_n\) is a bijection. 
\end{proof}

We now summarise some known results related to counting by Frobenius number proved by Backelin~\cite{Backelin1990MS-number}, and stated at the beginning of his article. 

\begin{proposition}\label{prop:Backelin-bounds} For any positive integer \(n\), 
	\(\frobcard{n} \in \left[2^{\left\lfloor {\frac{n-1}{2}} \right\rfloor}, 4\cdot 2^{\left\lfloor {\frac{n-1}{2}} \right\rfloor} \right)\).
\end{proposition}

\begin{proposition}\label{prop:Backelin-asymptotic}\cite[Proposition 1]{Backelin1990MS-number}
	The following limits exist:
	\[ \lim_{n \text{ odd}} 2^{-\frac{n}{2}}\frobcard{n} \: ,\quad \lim_{n \text{ even}} 2^{-\frac{n}{2}}\frobcard{n}.\]
	
	In other words, the sequences \((\frobcard{2n})\) and \((\frobcard{2n+1})\) grow asymptotically as \(c_0 2^{n}\) and \(c_1 2^{n}\) respectively, where \(c_0\) and \(c_1\) are some positive real numbers.
\end{proposition}

Backelin also proves the following result on the distribution of multiplicities of numerical semigroups in~\(\frobset{n}\).

\begin{proposition}\label{prop:Backelin-distribution}\cite[Proposition 2]{Backelin1990MS-number}
  For any real number \(\varepsilon >0\), there exists \(C\in \mathbb{N}\) such that for every \(n\in\mathbb{N}\)
\[ \left|  \left\lbrace  S \in \frobset{n} : \left| \multiplicityoper(S)- \frac{n}{2}\right| > C \right\rbrace  \right|   < \varepsilon 2^{\frac{n}{2}}.\]
\end{proposition}
Note that for a given \(\varepsilon>0\) and the corresponding \(C\) in the above theorem we have that for most numerical semigroups \(S\) in \(\frobset{n}\) , the multiplicity \(\multiplicityoper(S) \) lies in \( (\frac{n}{2}-C, \frac{n}{2}+C)\). In particular, if \(n\) is large enough compared to \(C\) then \(\multiplicityoper(S) \) lies in \( (\frac{n}{3}, n)\). Therefore
Proposition~\ref{prop:Backelin-distribution} implies that the proportion of numerical semigroups in \(\frobset{n}\) whose depth \(\lceil (n+1)/\multiplicityoper(S) \rceil\) is equal to either \(2\) or \(3\) tends to \(1\) as \(n\rightarrow \infty\). A similar result in the context of counting by genus was conjectured by Zhao~\cite{Zhao2010SF-Constructing} and later proved by Zhai \cite{Zhai2012SF-Fibonacci}.

\begin{proposition}\cite{Zhai2012SF-Fibonacci}\label{prop:Zhai_depth_less_than_3}
	The proportion of numerical semigroups \(S\) with genus \(g\) that satisfy the condition \(\Frobeniusoper(S)<3\multiplicityoper(S)\) tends to \(1\) as \(g\rightarrow\infty\).
\end{proposition}

Kaplan and Singhal \cite{KaplanSinghal2023ECA-expected} established another asymptotic result under counting by genus.

\begin{proposition}\cite[Corollary~7]{KaplanSinghal2023ECA-expected}\label{prop:kaplan_singhal} 
	The proportion of numerical semigroups \(S\) with genus \(g\) that satisfy the inequality \( |\primitivesoper(S)|\geq \frac{\multiplicityoper(S)}{2}\) tends to \(1\) as \(g\rightarrow\infty\).
\end{proposition} 

We finally collect some results related to Wilf's conjecture.

\begin{theorem}\cite[Theorem 24]{Kaplan2012JPAA-Counting}\label{th:Kaplan_genus_Wilf}
	Let \(S\) be a numerical semigroup with genus \(g\) and multiplicity \(m\). If \(2g<3m\) then \(S\) satisfies Wilf's conjecture.
\end{theorem}

\begin{theorem}\cite{Eliahou2018JEMS-Wilfs}\label{th:depth3-are-wilf}
	Numerical semigroups \(S\) with \(\Frobeniusoper(S)<3 \multiplicityoper(S)\) satisfy Wilf's conjecture.
\end{theorem}
Sammartano showed in \cite{Sammartano2012SF-Numerical} that the condition \(\frac{\multiplicityoper(S)}{2}\leq |\primitivesoper(S)|\) is sufficient for a numerical semigroup \(S\) to satisfy Wilf's conjecture. This was later improved by Eliahou through the following result.

\begin{theorem}\cite{Eliahou2020EJC-graph}\label{th:embedd_dim_geq_m_by_3}
	Numerical semigroups \(S\) with \( |\primitivesoper(S)|\geq\frac{\multiplicityoper(S)}{3}\) satisfy Wilf's conjecture.
\end{theorem}

Therefore, as observed by Kaplan and Singhal, this result together with Proposition \ref{prop:kaplan_singhal} implies the following corollary.

\begin{corollary}\label{cor:asymp_wilf_genus}
	The proportion of numerical semigroups with genus \(g\) that satisfy Wilf's conjecture tends to \(1\) as \(g\rightarrow\infty\).
\end{corollary}
 Note that the above corollary can also be obtained from Proposition \ref{prop:Zhai_depth_less_than_3} and Theorem~\ref{th:depth3-are-wilf}. Similarly,  Proposition~\ref{prop:Backelin-distribution} together with Theorem~\ref{th:depth3-are-wilf} yields the following corollary.

\begin{corollary}\label{cor:asymp_wilf_frob}
	The proportion of numerical semigroups in \(\frobset{n}\) that satisfy Wilf's conjecture tends to \(1\) as \(n\rightarrow\infty\).
\end{corollary}

We next recall a well known result on the divisor bound (see \cite[Theorem 317]{HardyWright2008Book-introduction}). We state it in asymptotic version.

\begin{lemma}\label{lemma:size_of_div_set}
	Let \(n\) be a positive integer and let \(\operatorname{d}(n)\) be the number of divisors of~\(n\). Then \(\operatorname{d}(n) = n^{o(1)} \) that is \(\frac{\log(d(n))}{\log n}\) tends to \(0\) as \(n\rightarrow\infty\).
\end{lemma}

\section{The map \protect{\mainmap}}\label{sec:mainmap}
In this section we introduce a map from \(\maxprimset{n}\) to \(\frobset{n}\) and study its properties.

\begin{definition}\label{def:Phi} 
  Given a positive integer \(n\), we define the map \(\Phi_n : \maxprimset{n} \to \frobset{n}\) by
  \[\Phi_n(S)=\left(S\setminus\{n\}\right)\cup(n,\infty).\] 
  When there is no risk of confusion we omit the subscript and write simply \(\Phi\) for \(\Phi_n\).
\end{definition}
  
\begin{lemma}\label{lemma:depths-Phi-relation}
	Let \(n\) be a positive integer and let \(S\in\maxprimset{n}\). Then the primitive depth of \(S\) coincides with the depth of \(\Phi(S)\). In other words \(\pdepth(S) = \depth(\Phi(S)).\)
\end{lemma}
\begin{proof}
	We observe that for a numerical semigroup \(S \neq \mathbb{N}\), \(\depth(S) = \left\lceil \frac{\Frobeniusoper(S)+1}{\multiplicityoper(S)} \right\rceil = \lceil\Frobeniusoper(S)/{\multiplicityoper(S)}\rceil\) since \(\multiplicityoper(S)\) does not divide \(\Frobeniusoper(S)\). It follows from the definitions that \(n=\max \primitivesoper(S)=\Frobeniusoper(\Phi(S))\) and \(\multiplicityoper(S)=\min \primitivesoper(S)=\multiplicityoper(\Phi(S))\). 
	Therefore
	\[\pdepth(S) = \left\lceil \frac{\max \primitivesoper(S)}{\min \primitivesoper(S)}\right\rceil = \left\lceil \frac{\Frobeniusoper(\Phi(S))}{ \multiplicityoper(\Phi(S))}\right\rceil =\depth(\Phi(S)).\]
	If \(S=\mathbb{N}\) then \(\Phi(S)=O_2\), and \(\pdepth(\mathbb{N})=1=\depth(O_2)\) as required.
\end{proof}

\begin{example}\label{example:phi-1-and-2}
	Note that \(\Phi_1(\mathbb{N})=O_2\). As \(\maxprimset{1}=\{\mathbb{N}\}\) and \(\frobset{1}=\{O_2\}=\{\langle 2,3\rangle\}\), \(\Phi_1\) is a bijection. We observe that 
	\(\maxprimset{2}=\emptyset\) and so \(\Phi_2\) is injective vacuously but not surjective.
\end{example}

Example~\ref{example:phi-1-and-2} treats the initial cases. Unless otherwise stated, we assume from now that \(n\ge 3\).  We next discuss some properties of the map \(\Phi:\maxprimset{n}\to \frobset{n}\). 

\begin{proposition}\label{prop:phi_basic_properties}
	Let \(n\ge 3\) be an integer and consider the map \(\Phi:\maxprimset{n}\to \frobset{n}\). The following assertions hold.
	\begin{enumerate}[label = \textnormal{(\roman*)}]
		\item  \(\Phi(\maxprimset{n})= \{S \in \frobset{n} : \gcd(\extleftsoper(S))=1 \}= \frobset{n}(1).\)\label{itm:image-Phi}
		\item \(O_{n+1} \in \frobset{n}\setminus \Phi(\maxprimset{n})\), in particular the map \(\Phi\) is not surjective.\label{itm:Phi_non_surjective}
		\item The map \(\Phi:\maxprimset{n}\to \frobset{n}\) is injective and so \(\Phi:\maxprimset{n}\to \frobset{n}(1)\) is a bijection.\label{itm:Phi-is-injective}
	\end{enumerate}
\end{proposition}
\begin{proof}
	We first consider part~\ref{itm:image-Phi}. We first show that \( \Phi(\maxprimset{n})\) is a subset of \( \frobset{n}(1)\). Let \(S \) be any element of \( \Phi(\maxprimset{n})\). Then \(\Frobeniusoper(S)\) is equal to \(n\) and there exists \(T \) in \( \maxprimset{n}\) such that \(S=\Phi(T)\). Therefore \[\primitivesoper(T)\subseteq T\cap[0,n]=\left(\Phi(T)\cup\{n\}\right)\cap [0,n] = \extleftsoper(S).\]
	Since \(\gcd(\primitivesoper(T))\) is equal to \(1\), \(\gcd(\extleftsoper(S))\) is equal to \(1\). Thus \(S\in \frobset{n}(1)\) and the claimed inclusion is established.
	
	It remains to show that \( \frobset{n}(1)\) is a subset of \( \Phi(\maxprimset{n})\).	Let \(R\) be an element of \( \frobset{n}(1)\) and let 
	\[S=\langle \extleftsoper(R)\rangle=\langle (R \cap [0,n])\cup\{n\}\rangle\in \maxprimset{n}.\]  Then \(\Phi(S)\) is equal to \(R\) and so \(R\) is an element of \( \Phi(\maxprimset{n})\), establishing the remaining claimed inclusion.
		
	We now consider part~\ref{itm:Phi_non_surjective}.	It is clear that \(O_{n+1}\) is an element of \( \frobset{n}\). Let \(T\) be any element of \(\maxprimset{n}\). Since \(\gcd (\primitivesoper(T)) \) is equal to \(1\) there exists at least one primitive \(t\) in \( \primitivesoper(T) \) such that \(0<t<n\). From the definition of the map \(\Phi\) it follows that \(t\in \Phi(T)\). As \(t\not\in O_{n+1} \) it follows that \(O_{n+1}\not\in \Phi(\maxprimset{n})\), concluding the proof of part~\ref{itm:Phi_non_surjective}.
		
	We finally consider part~\ref{itm:Phi-is-injective}. Let \(S\) and \(R\) be elements of \(\maxprimset{n}\) such that \(\Phi(S)=\Phi(R)\). Then \(\left(S\setminus\{n\}\right)\cup(n,\infty)
	= \left(R\setminus\{n\}\right)\cup(n,\infty)\) and, since \(n\in S\cap R \), we have \(S\cap [0,n]=R\cap [0,n].\) In particular \(\primitivesoper(R) \subseteq S\) and \(\primitivesoper(S) \subseteq R\). Therefore \(\primitivesoper(S)=\primitivesoper(R)\) and \(S=R\). Therefore the map \(\Phi:\maxprimset{n}\to \frobset{n}\) is injective. Part~\ref{itm:image-Phi} now gives that \(\Phi:\maxprimset{n}\to \frobset{n}(1)\) is a bijection.
\end{proof}

\section{A correspondence between two sequences}\label{sec:Nf-Af}
In this section we relate the cardinalities \(\maxprimcard{n}\) of \(\maxprimset{n}\) and \(\frobcard{n}\) of \(\frobset{n}\) and prove Theorems \ref{th:frob_in_terms_of_max_prim} and \ref{th:asymptotics_maxprim_counting}. 

Recalling Definition~\ref{notation:Nf-d} note that

\begin{equation}\label{eq:Nf(d)}
	\frobcard{n}=\sum_{d|n}\frobcard{n}(d).
\end{equation}

 Proposition~\ref{prop:phi_basic_properties} yields \(\maxprimcard{n}= |\Phi(\maxprimset{n})| = \frobcard{n}(1) \). Also, since \(\maxprimset{1}=\{\mathbb{N}\}\) and \(\frobset{n}(n) = \{ O_{n+1}\}\), \(\frobcard{n}(n) = 1 = \maxprimcard{1} \).

\begin{lemma}\label{lemma:B(d)}
	Let \(n>2\) be an integer and let \(d\) be a divisor of \(n\). Then
    \[ \frobcard{n}(d) = \maxprimcard{\frac{n}{d}} .\] 
\end{lemma}
\begin{proof}
	By Proposition~\ref{prop:phi_basic_properties} the map \( \Phi_{\frac{n}{d}}: \maxprimset{\frac{n}{d}} \to \frobcard{\frac{n}{d}}\) is a bijection. Moreover by Lemma~\ref{lemma:maps-delta-injective} the map \(\delta_{n,d}^{-1}: \frobset{\frac{n}{d}}(1) \to \frobset{n}(d)\) is a bijection. Therefore the composition \(\delta_{n,d}^{-1} \circ \Phi_{\frac{n}{d}} : \maxprimset{\frac{n}{d}}  \to \frobset{n}(d)\) is also a bijection and the result follows. 
\end{proof}

We can now prove Theorem~\ref{th:frob_in_terms_of_max_prim}.

\noindent\textit{Proof of Theorem~\ref{th:frob_in_terms_of_max_prim} }. For \(n\) equal to \(1\), we have \(\frobcard{1}=\maxprimcard{1}=1\), and for \(n\) equal to \(2\) we have \(\frobcard{2} = 1 = \maxprimcard{1}+\maxprimcard{2}\) since \( \maxprimcard{2}=0\) (see Table \ref{table:counting_by_maxprim-and-frob}). For \(n>2\) and for every divisor \(d\) of \(n\), (\ref{eq:Nf(d)}) and Lemma~\ref{lemma:B(d)} yield
	\[\frobcard{n} = \sum_{d| n} \maxprimcard{\frac{n}{d}}.\]
	For a prime \(p\), this gives \(\frobcard{p}= \maxprimcard{p}+\maxprimcard{1}= \maxprimcard{p} + 1\).  \qed \\

Using the well known Möbius inversion formula we now show that the sequence \((\maxprimcard{n})\) is determined by the sequence \((\frobcard{n})\). For completeness we give below the definition of the Möbius function as well as the Möbius inversion formula.

\begin{definition}\label{def:Mobius_map}
	The Möbius function \(\mu : \mathbb{N} \to \{-1,0,1\}\) is defined as follows. Given a nonnegative integer \(n\), \(\mu(n)\) is the sum of all \(n\)-th roots of unity. More explicitly
	
	\begin{equation*}
		\mu(n) = \begin{cases}
			 \text{  } 1 & \text{if } n=1\\
			 \text{  } 0 & \text{if } n \text{ has a factor which is a square} \\
			 (-1)^k & \text{if } n \text{ is square-free with $k$ distinct prime factors.} 
		\end{cases}
	\end{equation*}
\end{definition}

\begin{theorem}\label{th:Mobius_inversion}~\cite[Theorem 266]{HardyWright2008Book-introduction}
	Let \(f\) and \(F: \mathbb{N} \to \mathbb{Z}\) satisfy
	\[F(n)=\sum_{d |n}f(d)\quad\text{for }n\geq 1 .\]
	Then
	\[ f(n)=\sum_{d |n}\mu\left(\frac{n}{d}\right)\cdot F(d)\quad\text{for }n\geq 1 .\]
	The functions \(f\) and \(F\) are said to be Möbius transforms of one another.
\end{theorem}

We can now prove Corollary~\ref{cor:dependence_of_Af_and_Nf}.

\noindent\textit{Proof of Corollary~\ref{cor:dependence_of_Af_and_Nf}}. 
	The result follows from Theorem~\ref{th:frob_in_terms_of_max_prim} and Theorem~\ref{th:Mobius_inversion} by taking \(f(n)\) to be \(\maxprimcard{n}\) and \(F(n)\) to be \(\frobcard{n}\).  \qed\\
	
We next study the growth of \(\maxprimcard{n}\) as a function of \(n\). We use Corollary \ref{cor:dependence_of_Af_and_Nf} to show that \(\maxprimcard{n}\) is asymptotically equal to \(\frobcard{n}\). We first prove Theorem~\ref{th:asymptotics_maxprim_counting}. %

\noindent\textit{Proof of Theorem~\ref{th:asymptotics_maxprim_counting}}. %
By Corollary \ref{cor:dependence_of_Af_and_Nf} for every positive integer \(n\)
\[ \maxprimcard{n} = \frobcard{n} + \sum_{d|n, d\neq n}\mu\left(\frac{n}{d}\right)\cdot \frobcard{d}. \]
We first observe that the sum of the terms under the summation sign is always negative because \(\maxprimcard{n} < \frobcard{n}\) due to injectivity of the map \(\Phi: \maxprimset{n} \to \frobset{n}\) (see Proposition \ref{prop:phi_basic_properties} \ref{itm:Phi-is-injective}). 
By Proposition \ref{prop:Backelin-bounds}, for large enough \(n\), the absolute value of the largest term under the summation sign is less than or equal to \(4\cdot 2^{\frac{n}{4}}\). 
Let \(\operatorname{d}(n)\) be the number of divisors of~\(n\). Then for large enough \(n\), 
\[\frobcard{n}-\maxprimcard{n} = \left|\sum_{d|n, d\neq n}\mu\left(\frac{n}{d}\right)\cdot \frobcard{d}\right| \leq \operatorname{d}(n) \cdot 4\cdot 2^{\frac{n}{4}}.  \]
Lemma~\ref{lemma:size_of_div_set} gives 
\( \frobcard{n}-\maxprimcard{n}  = n^{o(1)} 2^{\frac{n}{4}}.
\)
\hfill \qed \\

We can now prove Corollary \ref{cor:subsequence_max_prim_count}.

\noindent\textit{Proof of Corollary~\ref{cor:subsequence_max_prim_count}}. 
By Theorem~\ref{th:asymptotics_maxprim_counting}, 
\[ \frac{\maxprimcard{n}}{\frobcard{n}} = 1-  \frac{n^{o(1)} 2^{\frac{n}{4}}}{\frobcard{n}}\]
which tends to \(1\) as \(n\rightarrow \infty\) because \(\frobcard{n}\geq 2^{\frac{n}{2}-1}\) by Proposition~\ref{prop:Backelin-bounds}. Proposition~\ref{prop:Backelin-asymptotic} now yields that the following limit exists
\[ \lim_{n \text{ even}} 2^{-\frac{n}{2}}\frobcard{n}=  \lim_{n\to\infty}\frac{\frobcard{2n}}{2^{\frac{2n}{2}}} = \lim_{n\to\infty}\frac{\maxprimcard{2n}}{2^{\frac{2n}{2}}}= \lim_{n \text{ even}} 2^{-\frac{n}{2}}\maxprimcard{n}. \]
Thus the sequence \( (\maxprimcard{2n}) \) grows asymptotically as \(c_0 2^{n}\) for some positive real number \(c_0\). 
A similar argument yields the corresponding result for the sequence  \( (\maxprimcard{2n+1})\). \qed \\	

We next discuss some bounds for \(\maxprimcard{n}\). By Proposition \ref{prop:phi_basic_properties}\ref{itm:Phi-is-injective}, for any integer \(n\ge 3\) the upper bound \(\maxprimcard{n}<\frobcard{n}\) holds. We obtain a lower bound for \(\maxprimcard{n}\) by an application of Bertrand's postulate~\cite[Theorem 418 and Note \S 22.3]{HardyWright2008Book-introduction}.

\begin{lemma}\label{lemma:Bertrand's_lower-bound-Af}
	For every integer \(n\geq8\), 
	\(\maxprimcard{n}\geq \frac{3}{4} \cdot 2^{ \left\lfloor \frac{n-1}{2} \right\rfloor}\).
\end{lemma}
\begin{proof}
	Fix an integer \(n\geq8\). By Bertrand's postulate, for any integer \(r>3\) the interval \((r,2r-2)\) contains a prime. Therefore there exists at least one prime \(p\) in \( \left( \left\lceil \frac{n}{2} \right\rceil,2\cdot \left\lceil \frac{n}{2} \right\rceil-2 \right) \subset \left( \frac{n}{2}, n-1\right)  \). Consider a set \(Z\) contained in \((\frac{n}{2}, n)\) that contains at least one element in \(\{p,n-1\}\). Then the set \( Z \cup \{n\} \) 
	has greatest common divisor equal to \(1\) and \(\langle Z \cup \{n\} \rangle \in \maxprimset{n}\). We thus obtain the following subset of~\(\maxprimset{n}\)
	\[ \left\lbrace \langle X\cup Y\cup \{n\} \rangle : X\subseteq \left( \frac{n}{2}, n\right) \setminus \{p, n-1\}, \emptyset \neq Y\subseteq \{p,n-1\}  \right\rbrace. \]
	The cardinality of the above set is  \(3\cdot 2^{ \left\lfloor \frac{n-1}{2} \right\rfloor-2}\) because in the above construction \(X\) can be chosen in \(  2^{ \left\lfloor \frac{n-1}{2} \right\rfloor-2} \) ways and \(Y\) in three ways, independently. This concludes the proof.
\end{proof}

We further observe that any set containing at least one pair of consecutive elements has its greatest common divisor equal to \(1\). In this way we obtain another lower bound for \(\maxprimcard{n}\) using an elementary combinatorial exercise that we state here without proof.
\begin{lemma}\cite[Exercise 30a]{Stanley2012Book-Enumerative}\label{lemma:elementary_comb}
	The number of subsets $S$ of the set $[n] = \{1,2,\ldots,n\} $ such that $S$ contains no two consecutive integers is $ F_{n+2}$, where \(F_k\) denotes the \(k\)-th Fibonacci number obtained with the condition \(F_1=F_2=1\).
\end{lemma}

This lemma implies that the number of subsets of \([n] \) that contain at least one pair of consecutive integers is \(2^{n}-F_{n+2}\). Moreover we observe that this result is invariant under translation of the set \([n]\), i.e., for every integer interval of size \(n\), the number of subsets that do not contain any two consecutive integers will remain unchanged. This gives us the following result. %

\begin{proposition}\label{prop:fibonacci_lower_bound}
  Let \(n\geq 3\) be an integer. Then	
	\[\maxprimcard{n} \geq 2^{\left\lfloor {\frac{n-1}{2}} \right\rfloor} - F_{\left\lfloor {\frac{n-1}{2}} \right\rfloor +1 }.\]
\end{proposition}
\begin{proof}
	We construct subsets of \((n/2,n]\) that contain \(n\) and have at least one pair of consecutive integers, and thus have their greatest common divisor equal to \(1\). For each set \(A\) with this property, we have a unique numerical semigroup \(\langle A \rangle \) in \(\maxprimset{n}\). Therefore the number of such subsets gives us a lower bound for \(\maxprimset{n}\). We consider two cases depending on whether such a subset contains \(n-1\) or not. We observe that if \(n-1\) is in such a set then it contains the pair \(\{n-1,n\}\) and thus its greatest common divisor is equal to \(1\). There are \(2^{\left\lfloor {\frac{n-3}{2}} \right\rfloor}\) distinct sets with this property. Otherwise if we have a set \(B\) contained in \((n/2,n]\) such that \(B\) contains \(n\) but not \(n-1\), and it contains at least one pair of consecutive integers, then \(B\) is the same as \(B' \cup \{n\} \) where \(B'\) is a subset of \((n/2,n-1)\) that has at least one pair of consecutive integers. Using Lemma~\ref{lemma:elementary_comb} for the integer interval \((n/2,n-1)\), the number of possibilities for \(B'\), and thus for \(B\), is \( 2^{\left\lfloor {\frac{n-3}{2}} \right\rfloor} - F_{\left\lfloor {\frac{n-3}{2}} \right\rfloor +2 }\). Hence 
	\[2^{\left\lfloor {\frac{n-3}{2}} \right\rfloor} + \left(  2^{\left\lfloor {\frac{n-3}{2}} \right\rfloor} - F_{\left\lfloor {\frac{n-3}{2}} \right\rfloor +2 }\right)  = 2^{\left\lfloor {\frac{n-1}{2}} \right\rfloor}- F_{\left\lfloor {\frac{n-1}{2}} \right\rfloor +1 } \leq \maxprimcard{n}.\qedhere\]
\end{proof}

It is straightforward to check that for \(n>10\) the lower bound for \(\maxprimcard{n} \) provided in Proposition~\ref{prop:fibonacci_lower_bound} is better than the one given in Lemma~\ref{lemma:Bertrand's_lower-bound-Af}.

\section{On Wilf's Conjecture}\label{sec:wilf}

In this section we discuss the analogous problem to Corollaries~\ref{cor:asymp_wilf_genus} and~\ref{cor:asymp_wilf_frob} under counting numerical semigroups by maximum primitive, namely: does almost every semigroup in \(\maxprimset{n}\) satisfy Wilf's conjecture as \(n\rightarrow\infty\)? We answer this question positively. In particular we prove Theorems \ref{th:Wilfs_conjecture_new_criterion} and \ref{th:Asymptotic_wilf_for_max_prim_counting}.

We first establish an analogous result to Proposition~\ref{prop:Backelin-distribution} on the distribution of multiplicities in \(\maxprimset{n} \).

\begin{proposition}\label{prop:Maxprim-distribution}
	For any real number \(\varepsilon >0\) there exists \(C\in \mathbb{N}\) such that for every \(n\in\mathbb{N}\)
	\[ \left|  \left\lbrace  S \in \maxprimset{n} : \left| \multiplicityoper(S)- \frac{n}{2}\right| > C \right\rbrace  \right|   < \varepsilon 2^{\frac{n}{2}} .\]
\end{proposition}
\begin{proof}
	Fix \(\varepsilon >0\). By Proposition \ref{prop:Backelin-distribution} there exists \(C\in \mathbb{N}\) such that for every \(n\in\mathbb{N}\)
	\[ \left|  \left\lbrace  S \in \frobset{n} : \left| \multiplicityoper(S)- \frac{n}{2}\right| > C \right\rbrace  \right|   < \varepsilon 2^{\frac{n}{2}} .\]
	
	For \(S\in \maxprimset{n}\), \(\multiplicityoper(S)= \multiplicityoper(\Phi(S))\) and \(n=\Frobeniusoper(\Phi(S))\) by definition of the map \(\Phi:\maxprimset{n}\to\frobset{n} \). Moreover the injectivity of \(\Phi\) yields  
	\begin{eqnarray*}
		\left|  \left\lbrace  S \in \maxprimset{n} : \left| \multiplicityoper(S)- \frac{n}{2}\right| > C \right\rbrace  \right| & \leq & \left|  \left\lbrace \Phi(S) \in \frobset{n} : \left| \multiplicityoper(\Phi(S))- \frac{\Frobeniusoper(\Phi(S))}{2}\right| > C \right\rbrace  \right|\\ & < & \left|  \left\lbrace S \in \frobset{n} : \left| \multiplicityoper(S)- \frac{n}{2}\right| > C \right\rbrace  \right| \\
		& < &   \varepsilon 2^{\frac{n}{2}}. 
		\qquad\qquad\qquad\qquad\qquad\qquad\qquad\qquad\qquad\qquad \qedhere
	\end{eqnarray*}
\end{proof}

In particular the above result implies that the proportion of numerical semigroups in \(\maxprimset{n}\) with primitive depth equal to either \(2\) or \(3\) tends to \(1\) as \(n\rightarrow\infty\). We make the following observation used repeatedly in later discussions.

\begin{remark}\label{remark:primitives_depth_two}
	For every numerical semigroup \(S\), every element of \( S \cap (\multiplicityoper(S),2 \multiplicityoper(S)) \) is a primitive of \(S\). Hence \(S \cap (\multiplicityoper(S),2 \multiplicityoper(S)) = \primitivesoper(S) \cap (\multiplicityoper(S),2 \multiplicityoper(S))\). 
\end{remark}
 
\begin{lemma}\label{lemma:left_primitives}
	Given any real number \(\varepsilon >0\) there exists \(M \in \mathbb{N}\) such that for every \(n>M\) 
	
	\[\left|\left\{S\in \maxprimset{n} : |S\cap (\multiplicityoper(S),2 \multiplicityoper(S))|< \sqrt{3\multiplicityoper(S)} \right\}\right| < \varepsilon \maxprimcard{n}. \]
\end{lemma}
\begin{proof}
	Given \(\varepsilon>0\) let \(C\) be as in Proposition \ref{prop:Maxprim-distribution}. If \(C\) is less than \(6\), we instead set \(C\) to be \(6\). Let \(n\) be greater than \(8 C^{2}\). Given \(x\in \mathbb{R} \) we abuse notation and write \(\binom{n}{x}\) for \(\binom{n}{\lfloor x \rfloor}\). Let
	\[Z = \left\{S\in \maxprimset{n} : |S\cap (\multiplicityoper(S),2 \multiplicityoper(S))|< \sqrt{3\multiplicityoper(S)} \right\}.\]
	We partition the set \(Z\) into three subsets and consider their sizes. We first consider
	\[U_1 = \left\{S\in Z : \left| \multiplicityoper(S)- \frac{n}{2}\right| > C \right\}. \] 
	By Proposition~\ref{prop:Maxprim-distribution}, \(|U_1|<\varepsilon 2^{\frac{n}{2}}\). By Lemma~\ref{lemma:Bertrand's_lower-bound-Af}, \(\maxprimcard{n}\geq   \frac{3}{4} \cdot 2^{ \left\lfloor \frac{n-1}{2} \right\rfloor } \) and in particular \(\maxprimcard{n}> \frac{1}{4}\cdot 2^{\frac{n}{2}} \). Hence \(|U_1|<4\varepsilon \maxprimcard{n}\).  
	
	We observe that for any numerical semigroup in \(Z\) the multiplicity \(m\) cannot be equal to~\(\frac{n}{2}\) because \(m\) cannot divide \(n\). Therefore \(Z\setminus U_1\) can be divided into two sets:
	\[ U_2 = \left\{S\in Z : \multiplicityoper(S) \in \left( \frac{n}{2}, \frac{n}{2} +C\right] \right\} \quad \text{and } \quad U_3 = \left\{S\in Z : \multiplicityoper(S)\in \left[ \frac{n}{2} - C, \frac{n}{2} \right) \right\}.\] 
	We claim that \(|U_2| <\varepsilon\maxprimcard{n}\). Note that every numerical semigroup \(S\) in \(U_2\) with multiplicity \(m\) is generated by a subset \(X\) of \([m,n]\) of size less than \(\sqrt{3m}+1\) such that \(X\) contains \(m\) and \(n\), and the greatest common divisor of \(X\) is equal to \(1\).  In particular, \( |U_2|\)  is bounded above by the number of subsets of \((m,n)\) of size less than \(\sqrt{3m}\). Hence  
	\[\left|U_2\right| \leq \sum_{m\in \left(  \frac{n}{2}, \frac{n}{2} + C \right] } \sum_{k<\sqrt{3m}} \binom{n-m}{k}.\]

	Let \(a,x,y\) be positive integers. Observe that if \(x\leq y< \frac{a}{2}\) then \(  \binom{a}{x} \leq  \binom{a}{y} \). Hence for every \(k<\sqrt{3m}\)
	\[
	\binom{n-m}{k} \leq  \binom{n-m}{\sqrt{2n}} 
	\]
	provided \( \sqrt{3m} \leq \sqrt{2n} < \frac{n-m}{2} \). One can check that \(n>8 C^{2}\) and \(C>2\) suffices for this to hold. Therefore 
	\[ \sum_{m\in \left(  \frac{n}{2}, \frac{n}{2} + C \right] } \sum_{k<\sqrt{3m}} \binom{n-m}{k} \leq \sum_{m\in \left(  \frac{n}{2}, \frac{n}{2} + C \right] } \sqrt{2n}\cdot \binom{n-m}{\sqrt{2n}} \leq C \cdot\sqrt{2n}\cdot \binom{n/2}{\sqrt{2n}}. \]

	One can verify that 
	\[	\lim_{n\to\infty} \frac{C \sqrt{2n}\cdot \binom{n/2}{\sqrt{2n}}}{2^{\frac{n}{2}}} = 0.\]
	Thus there exists a positive integer \(M_0\) such that for every \(n>M_0\) we have
	\[ C \sqrt{2n}\cdot \binom{n/2}{\sqrt{2n}} < \frac{\varepsilon}{4}\cdot 2^{\frac{n}{2}}.\]
	By Lemma \ref{lemma:Bertrand's_lower-bound-Af}, \(\maxprimcard{n}> \frac{1}{4}\cdot 2^{\frac{n}{2}}\). Hence for every integer \(n\) greater than \(\max \{8C^{2}, M_0\} \),
	\[ 	\left| U_2 \right| \leq C \sqrt{2n}\cdot \binom{n/2}{\sqrt{2n}} < \frac{\varepsilon}{4}\cdot 2^{\frac{n}{2}} < \varepsilon\maxprimcard{n}.\]
	
	We next consider \(U_3\) and claim that \(|U_3|<  \varepsilon \maxprimcard{n}.\) %
	Note that every numerical semigroup~\(S\) in \(U_3\) with multiplicity \(m\) is generated by a subset \(X\) of \([m,n]\) such that  \(X\) contains \(m\) and~\(n\), the size of \(X\cap (m,2m)\) is less than \(\sqrt{3m}\), \(X\cap (2m,n)\) does not contain the sum of any two elements of \(X\), and the greatest common divisor of \(X\) is equal to 1. In particular \(|U_3|\) is bounded above by the product of the number of subsets of \((m,2m)\) of size less than \(\sqrt{3m}\) by the number of subsets of \((2m,n]\). Hence  
	\[\left|U_3\right| \leq \sum_{m\in \left[ \frac{n}{2}-C, \frac{n}{2}\right) } \sum_{k<\sqrt{3m}} \binom{m}{k}\cdot 2^{n-2m}.\]
	and arguing as before with \(n>2C^2\) and \(C>5\), we obtain 
	\[ \left|U_3\right| \leq C\sqrt{2n}\cdot \binom{n/2}{\sqrt{2n}}\cdot 2^{2C}.\]
	Thus, arguing as before, we obtain a positive integer \(M_1\) such that for all \(n>M_1\)
	\[	|U_3|\leq C \sqrt{2n}\cdot \binom{n/2}{\sqrt{2n}} \cdot 2^{2C} < \frac{\varepsilon}{4}\cdot 2^{\frac{n}{2}} < \varepsilon \maxprimcard{n}.
	\]
	
	Let \(M=\max\{M_0,M_1, 8C^{2}\}\). Then for every integer \(n>M\) 
	\[ |Z| = |U_1|+|U_2|+|U_3| < 6\varepsilon \maxprimcard{n}. \]
	Taking \(\varepsilon'= \frac{\varepsilon}{6}\) instead of \(\varepsilon\) in the above proof, we obtain a positive integer \(M\) such that \(|Z| < \varepsilon \maxprimcard{n}\) for every integer \(n>M\). %
	\end{proof}

\begin{corollary}\label{cor:asymptotic_dense_contition}
	The proportion of numerical semigroups \(S\in \maxprimset{n}\) that satisfy the inequality \(|\primitivesoper(S)|\geq\sqrt{3\multiplicityoper(S)}\) tends to \(1\) as \(n\rightarrow\infty\).
\end{corollary}
\begin{proof}
	This follows from Lemma~\ref{lemma:left_primitives} and Remark~\ref{remark:primitives_depth_two}.
\end{proof}

We can now prove Theorem~\ref{th:Wilfs_conjecture_new_criterion}.

\noindent\textit{Proof of Theorem~\ref{th:Wilfs_conjecture_new_criterion}}.  Let \(S\) be a numerical semigroup with multiplicity~\(m\) and depth~\(d\). Without loss of generality \(d\geq 4\) as otherwise \(S\) satisfies Wilf's conjecture by Theorem~\ref{th:depth3-are-wilf}. Note that \(dm\geq \Frobeniusoper(S)+1\). Suppose that 
\[|S\cap (m,2m)|\geq \sqrt{\frac{dm}{d-2}} \]
and thus, by Remark \ref{remark:primitives_depth_two}, we have \(|\primitivesoper(S)\cap (m,2m)|\geq \sqrt{\frac{dm}{d-2}}\). %
For every \(a\) in \(\primitivesoper(S)\cap (m,2m)\),
	\[\{a, a+m, \ldots, a+(d-3)m\} \subset \leftsoper(S). \]
	Therefore \(|\leftsoper(S)|\geq (d-2)|\primitivesoper(S)\cap (m,2m)|.\) Hence
	\[|\primitivesoper(S)|\cdot|\leftsoper(S)|\geq |\primitivesoper(S)|\cdot (d-2)|\primitivesoper(S)\cap (m,2m)| \geq (d-2)\left( \frac{dm}{d-2}\right)  = dm \geq \Frobeniusoper(S)+1 .\]
	Thus \(S\) satisfies Wilf's conjecture as required.
	
	Since \(d> 3\), \(\frac{dm}{d-2}\leq 2m\) and so a numerical semigroup \(S\) with \(|S\cap (m,2m)|\geq \sqrt{2m}\) also satisfies Wilf's conjecture.
          Moreover if \(d\leq 3\) and \(|S\cap (m,2m)|\geq \sqrt{2m}\), the result follows from Theorem~\ref{th:depth3-are-wilf}.
        \qed 
\\	 

We now prove Corollary \ref{cor:improvement_on_Wilf_2g_3m_criterion}.

\noindent\textit{Proof of Corollary~\ref{cor:improvement_on_Wilf_2g_3m_criterion}}. 

Let \(d\) be the depth of the numerical semigroup \(S\). Without loss of generality \(d\geq 4\) as otherwise \(S\) satisfies Wilf's conjecture by Theorem~\ref{th:depth3-are-wilf}. Note that \(S\) has \(m-1\) gaps in the interval \((0,m)\) and, since \(d>3\), at least two gaps, namely \(\Frobeniusoper(S)\) and \(\Frobeniusoper(S)-m\), in \((2m,\infty)\). Therefore \(S\) has at least \(m+1\) gaps in \(\mathbb{N}\setminus (m,2m) \). Suppose that \(g<2m-\sqrt{2m}\). Then \(S\) has at most \(m-1-\sqrt{2m}\) gaps in \((m,2m)\), and so \(S\) has at least \(\sqrt{2m}\) elements in \((m,2m)\). The result now follows from Theorem~\ref{th:Wilfs_conjecture_new_criterion}. \qed%
\\
	
	Observe that the above corollary provides an improvement to Theorem~\ref{th:Kaplan_genus_Wilf}. 
We next prove Theorem~\ref{th:Asymptotic_wilf_for_max_prim_counting}.

\noindent\textit{Proof of Theorem~\ref{th:Asymptotic_wilf_for_max_prim_counting}}. This follows from  Lemma \ref{lemma:left_primitives} and Theorem \ref{th:Wilfs_conjecture_new_criterion}. \qed \\

The class of numerical semigroups considered in Theorem~\ref{th:Wilfs_conjecture_new_criterion} is therefore \emph{large} in the sense of Corollary \ref{cor:asymptotic_dense_contition}. Recall that the class of numerical semigroups that satisfy \( |\primitivesoper(S)| \geq \multiplicityoper(S)/3\) and the class of numerical semigroups that satisfy \(\depth(S)\leq3\) are known to satisfy Wilf's conjecture (see Theorems~\ref{th:embedd_dim_geq_m_by_3} and~\ref{th:depth3-are-wilf}) and these classes are large in the sense of Propositions~\ref{prop:kaplan_singhal} and \ref{prop:Zhai_depth_less_than_3}, respectively.
For the sake of non triviality, we give below an argument showing that the class considered in Theorem~\ref{th:Wilfs_conjecture_new_criterion} differs from the latter two classes.

Consider the integer intervals \(I=[m,\frac{4m}{3}-1)\) and \(I'=[m,m+m^{\frac{1}{2}}-1)\) for some large enough value of \(m\). 
We may construct coprime subsets of \(I\) that contain at least \(\sqrt{2m}+1\) elements including \(m\). The numerical semigroups \(S\) generated by such subsets satisfy Wilf's conjecture by Theorem~\ref{th:Wilfs_conjecture_new_criterion} and \(|\primitivesoper(S)|\leq |I| <m/3\). Moreover \(S\) is contained in the numerical semigroup \( \langle I \rangle\) and so \(\Frobeniusoper(S)\geq\Frobeniusoper(\langle I \rangle)\). By \cite[Corollary~5]{Garcia-SanchezRosales1999PJM-Numerical} we have \(\Frobeniusoper(\langle I \rangle) = \left\lceil \frac{m-1}{\left\lfloor m/3-1 \right\rfloor}\right\rceil m -1 = 4m-1>3m.\) Hence \(\Frobeniusoper(S)>3m\) and so \(\depth(S)\geq 4\), as required.\smallskip 

We further observe that since \(I'\) is contained in \(I\) the numerical semigroups \(S'\) generated by coprime subsets of \(I'\) containing \(m\) satisfy \(\depth(S')\geq 4\). Moreover \(|\primitivesoper(S')|<m^{\frac{1}{2}}<\sqrt{\frac{dm}{d-2}}<\frac{m}{3}\) for all \(d> 2\). Therefore \(S'\) is not contained in the classes of numerical semigroups considered in Theorems~\ref{th:Wilfs_conjecture_new_criterion}, ~\ref{th:embedd_dim_geq_m_by_3} and~\ref{th:depth3-are-wilf}. We note that the maximum primitive of \(S'\) is smaller than \(2m\), and therefore \(S'\) has primitive depth 
equal to \(2\). Moreover we observe that the class of numerical semigroups with primitive depth \(2\) contains numerical semigroups \(S\) for which \(|\primitivesoper(S)|\) is smaller than \(\multiplicityoper(S)^{\frac{1}{n}}\) for any given positive integer \(n\). We therefore conclude the article by posing Problem~\ref{prob:primitve_depth_2_wilf} to draw attention to Wilf's conjecture for this class. 
\medskip

 \textbf{Declaration of AI use:} The authors declare that no artificial intelligence (AI) tools were used in the research, analysis, computation, or preparation of this manuscript.
	
%
\printbibliography
\end{document}